\documentclass[reqno]{amsart}

\usepackage[T1]{fontenc} 
\usepackage{amsmath,amssymb,amsthm}
\usepackage{mathrsfs} 
\usepackage{array} 
\usepackage{url} 
\usepackage{listings}
\usepackage{longtable}

\usepackage{breakurl}
\usepackage[breaklinks]{hyperref}
\usepackage{float}
\usepackage[shortlabels]{enumitem}
\restylefloat{table}
\floatstyle{plaintop}
\restylefloat{table}
\usepackage{tikz,tikz-3dplot}
\usetikzlibrary{positioning}
\usepackage{caption}
\usepackage{braket}
\usepackage{logunit}
\usepackage{mathtools}

\newtagform{alph}[]()
\newtagform{Alph}[]()
\newtagform{roman}[]()
\newtagform{scroman}[][]

\title[Geometry of Biquadratic and Cubic Log-Unit Lattices]{%
Geometry of Biquadratic and Cyclic Cubic Log-Unit Lattices}

\author[Azpeitia Tellez]{Fernando Azpeitia Tellez$^1$}
\address{$^1$California State University San Marcos}
\email{azpei002@cougars.csusm.edu}

\author[Powell]{Christopher Powell$^2$}
\address{$^2$California State University San Marcos}
\email{powel054@cougars.csusm.edu}

\author[Sharif]{Shahed Sharif$^3$}
\address{$^3$California State University San Marcos}
\email{ssharif@csusm.edu}
\begin{document}
\maketitle

\section{Introduction}
\label{sec:introduction}

By Dirichlet's Unit Theorem, given a number field, the logarithm of the units in the ring of integers forms a lattice, called the \emph{log-unit lattice}. We investigate the geometry of the log-unit lattices associated to biquadratic and cyclic cubic number fields. Consider $\R^n$ as a vector space with the usual Euclidean inner product. We say a lattice embedded in $\R^n$ is \emph{orthogonal} if it has a basis of pairwise orthogonal vectors, where orthogonal means with respect to the usual inner product. Then we have the following:

\begin{theorem}\label{thm:typeI-primes}
    Suppose $K=\Q(\sqrt{p_1},\sqrt{p_2})$, where $p_1, p_2$ are distinct primes. Let $K_i = \Q(\sqrt{p_i})$ for $i=1,2$, and $K_3= \Q(\sqrt{p_1p_2})$. Let $\epsilon_i$ be a fundamental unit for $K_i$ and $e_i = \Nm_{K_i/\Q} (\epsilon_i)$. If either
  \begin{enumerate}
      \item $p_1 \equiv 1 \pmod{4}$, $p_2 \not\equiv 3 \pmod{4}$, and $e_1 = e_2 = -1$, $e_3 = 1$; or
      \item $p_1 \equiv 3 \pmod{4}$ and $p_2 = 2$,
\end{enumerate}
then the log-unit lattice of $K$ is orthogonal.
\end{theorem}

\begin{theorem}\label{thm:typeII-primes}
  If $K = \Q(\sqrt{p_1}, \sqrt{p_2})$ where $p_1,p_2$ are distinct primes satisfying $p_1 \equiv p_2 \equiv 3 \pmod{4}$, then the log-unit lattice of $K$ is not orthogonal.
\end{theorem}



\begin{theorem}\label{thm:typeiv-primes}
Suppose $K=\Q(\sqrt{p_1},\sqrt{p_2})$, where $p_1, p_2$ are distinct primes. 
  Let $\epsilon_i$ be a fundamental unit for $K_i := \Q(\sqrt{p_i})$, $K_3 = \Q(\sqrt{p_1p_2})$ and let $e_i = \Nm_{K_i/\Q} (\epsilon_i)$. If $p_1 \equiv 1 \pmod{4}$, $p_2 \not\equiv 3 \pmod{4}$, and $e_i=-1$ for all $i$,
then the log-unit lattice of $K$ is not orthogonal.
\end{theorem}

In the case where the log-unit lattice is orthogonal, we can estimate the size of an appropriate fundamental domain.
\begin{corollary}\label{cor:typeI-covering-radius-bounds}
Suppose $K = \Q(\sqrt{p_1},\sqrt{p_2})$, where $p_1, p_2$ are distinct primes. Let $\Delta_1, \Delta_2, \Delta_3$ be the discriminants of $K_1=\Q(\sqrt{p_1}), K_2=\Q(\sqrt{p_2})$, and $K_3=\Q(\sqrt{p_1p_2})$, respectively. For each $i$, let 
\[
\tilde{\Delta}_i = \log \left( \frac{1}{2}(\sqrt{\Delta_i-4}+\sqrt{\Delta_i})\right) \textrm{ and } \hat{\Delta}_i = \sqrt{\frac{1}{2}\Delta_i}\left(\frac{1}{2}\log \Delta_i + 1\right).
\]
If hypothesis $(1)$ of Theorem~\ref{thm:typeI-primes} holds, then the log-unit lattice of $K$ has a fundamental domain which is an orthogonal box with dimensions $\ell_1,\ell_2,\ell_3$, where
\begin{alignat*}{2}
  2\tilde{\Delta}_1 &\leq \ell_1 & &< 2\hat{\Delta}_1 \\
  2\tilde{\Delta}_2 &\leq \ell_2 & &< 2\hat{\Delta}_2, \textrm{ and} \\
  \tilde{\Delta}_3 &\leq \ell_3 & &< \hat{\Delta}_3.
\end{alignat*}
If hypothesis $(2)$ of Theorem~\ref{thm:typeI-primes} holds, then 
the log-unit lattice has a fundamental domain which is an orthogonal box with dimensions $\ell_1,\ell_2,\ell_3$, where
\begin{alignat*}{2}
  \tilde{\Delta}_1 &\leq \ell_1 & &< \hat{\Delta}_1 \\
  2\tilde{\Delta}_2 &\leq \ell_2 & &< 2\hat{\Delta}_2, \textrm{ and} \\
  \tilde{\Delta}_3 &\leq \ell_3 & &< \hat{\Delta}_3.
\end{alignat*}
\end{corollary}

Using the above, we obtain estimates for the covering radius of the lattice. Recall that the covering radius $\rho$ of a lattice $\Lambda$ is the radius of the largest open ball in $\Lambda \otimes \R$ which does not contain any elements of $\Lambda$.
\begin{corollary}
  With notation and hypotheses as in Corollary~\ref{cor:typeI-covering-radius-bounds}, let $\rho$ be the covering radius of the log-unit lattice of $K$. If hypothesis $(1)$ of Theorem~\ref{thm:typeI-primes} holds, then
  \[
    \left(\tilde{\Delta}_1^2 + \tilde{\Delta}_2^2 + \frac{1}{4}\tilde{\Delta}_3^2\right)^{\frac{1}{2}} \leq \rho < \left(\hat{\Delta}_1^2 + \hat{\Delta}_2^2 + \frac{1}{4}\hat{\Delta}_3^2\right)^{\frac{1}{2}}.
  \]
If hypothesis $(2)$ of Theorem~\ref{thm:typeI-primes} holds, then
  \[
    \left(\frac{1}{4}\tilde{\Delta}_1^2 + \tilde{\Delta}_2^2 + \frac{1}{4}\tilde{\Delta}_3^2\right)^{\frac{1}{2}} \leq \rho < \left(\frac{1}{4}\hat{\Delta}_1^2 + \hat{\Delta}_2^2 + \frac{1}{4}\hat{\Delta}_3^2\right)^{\frac{1}{2}}.
  \]
\end{corollary}

\begin{proof}
Let $F$ be a fundamental domain which is an orthogonal box with dimensions $\ell_1,\ell_2,\ell_3$. Then the largest sphere centered in $F$ which does not contain any of the vertices must have radius $\sqrt{(\frac{\ell_1}{2})^2 + (\frac{\ell_2}{2})^2 + (\frac{\ell_3}{2})^2}$. The claim follows from Corollary~\ref{cor:typeI-covering-radius-bounds}.
\end{proof}
Theorems~\ref{thm:typeI-primes}, \ref{thm:typeII-primes}, and \ref{thm:typeiv-primes} are proven as special cases of a stronger result, Theorem~\ref{thm:all-types}. The latter theorem is based on a classification of the units of real biquadratic fields into four types, and shows that only one of these types has an orthogonal log-unit lattice.

Lastly, in Theorem~\ref{thm:cubic-cyclic}, we show that the log-unit lattices of cyclic cubic fields are all similar; that is, they are isometric up to scaling.

The motivation behind the above results stems from lattice-based cryptography. The security of a number of recent cryptosystems depends upon a given log-unit lattice having a suitable geometry---generally, one for which lattice algorithms like LLL perform poorly. For example, in~\cite{cramer2016recov}, it is shown that the logs of the cyclotomic units, which form a sublattice of the log-unit lattice, have a dual basis which is, up to scaling, ``close'' to orthonormal. As observed in~\cite[\S{}1]{cramer2016recov}, a consequence is that the SOLILOQUY cryptosystem~\cite{campbell2014soliloquy} and the Smart-Vercauteren cryptosystem of~\cite{smart2010fhe} are broken.

Unfortunately, aside from~\cite{cramer2016recov}, \cite{stickelberger}, and~\cite{ducas-plancon-wesolowski}, there are no known results on the geometry of log-unit lattices. The goal of this paper is to begin such a study  by considering the log-unit lattices for low-degree number fields. The log-unit lattices of both quadratic number fields and complex biquadratic number fields have rank at most $1$, and hence have uninteresting geometry. For real biquadratic number fields, a result of Kubota~\cite[Satz 1]{kubota1956uber} allows us to determine geometric properties of the log-unit lattices, specifically orthogonality. Lastly, the Galois action imposes strict conditions on the log-unit lattice of cyclic cubic number fields.

Mazur-Ullom~\cite{mazur2004biquadratic} classify real biquadratic number fields into 4 types. We will deal with each type in order, starting in \S{}\ref{sec:type-i-lattices}, before turning to cyclic cubic fields in \S{}\ref{sec:cubic-cycl-latt}.

\paragraph{\textbf{Acknowledgments.}}

This project was partially supported by the NSF grant DMS-1247679 and by Viasat. The authors are grateful to Pete Clark, Kamel Haddad, and especially Wayne Aitken for helpful comments.

\section{Preliminaries}
\label{sec:preliminaries}

\subsection{Lattices}
\label{sec:lattices}

Recall that a \emph{lattice} $\Lambda$ is a discrete subgroup of $\R^n$ for some $n$. 
If $b_1, \dots, b_m \in \R^n$ is a linearly independent list of vectors, let
\[
  \Lambda(b_1, \dots, b_m) = \Z b_1 + \cdots + \Z b_m.
\]
We say $\Lambda(b_1, \dots, b_m)$ is the lattice \emph{generated} by the $b_i$, or the lattice with \emph{basis} $b_1, \dots, b_m$.
We also say that such a lattice has \emph{rank $m$}. If a lattice has an orthogonal basis, we say that the lattice is \emph{orthogonal}. Write $\Lambda_\R$ for $\Lambda \otimes \R \subset \R^n$, so that $\Lambda(b_1, \dots, b_m)_\R$ is the $\R$-span of the $b_i$. Let $\braket{v,w}$ denote the Euclidean inner product of $v,w \in \R^n$.

\begin{definition}
Let $\Lambda \subset \R^n$ be a rank $m$ lattice. For $r > 0$, let $B_r \subset \R^n$ be the closed ball of radius $r$ centered at the origin. Then for $i = 1, \dots, m$, define $\lambda_i(\Lambda)$ by
\[
  \lambda_i(\Lambda) = \min\{r \mid \dim (\mathrm{Span}(\Lambda \cap B_r)) \geq i \}.
\]
\end{definition}
Thus for example $\lambda_1(\Lambda)$ is the length of the shortest nonzero vector in $\Lambda$, and $\lambda_m(\Lambda)$ is the smallest number such that there exists a basis $b_1, \dots, b_m$ of $\Lambda$ for which $\|b_i\| \leq \lambda_m$ for all $i$.

\begin{lemma}\label{lemma:lambda1-b1}
  Suppose $b_1, \dots, b_m$ is an orthogonal basis for a lattice $\Lambda$, and $\|b_1\| \leq \|b_i\|$ for all $i$. Then $\lambda_1(\Lambda) = \|b_1\|$.
\end{lemma}

\begin{proof}
  If $v \in \Lambda$, then there are integers $\alpha_1, \dots, \alpha_m$ such that $v = \sum \alpha_i b_i$. But then
  \[
    \|v\|^2 = \sum \alpha_i^2 \|b_i\|.
  \]
  Thus if $v \neq 0$, then $\|v\| \geq \|b_1\|$. The claim follows.
\end{proof}

\begin{proposition}\label{prop:one-vector-orthogonal}
  Suppose $V$ be a subspace of $\R^n$. Let $v_0 \in V$ be a nonzero vector, and let $W$ be the orthogonal complement of $\spn(v_0)$ in $V$. Let $L_W \subset W$ be a lattice, and let $L$ be the lattice generated by $L_W$ and $v_0$. Suppose that $(b_1, \dots, b_n)$ is an orthogonal basis for $L$. Then $v_0 = \pm b_i$ for some $i$.
\end{proposition}

\begin{proof}
  Without loss of generality, the $b_i$ are ordered so that $\|b_i\| \leq \|b_{i+1}\|$ for each $i$. Choose the largest $k$ for which $\|b_k\| \leq \|v_0\|$. Since $\|b_1\| = \lambda_1 \leq \|v_0\|$, such a $k$ exists. As the $b_i$ form a basis for $L$, there exist $\alpha_i \in \Z$ for which
  \[
    v_0 = \sum \alpha_i b_i = \sum_{i = 1}^k \alpha_i b_i + \sum_{j=k+1}^n \alpha_j b_j.
  \]
  Hence
  \[
    \|v_0\|^2 = \sum_{i=1}^k \alpha_i^2 \|b_i\|^2 + \sum_{j=k+1}^n \alpha_j^2 \|b_j\|^2.
  \]
  But for $j \geq k+1$, $\|b_j\|^2 > \|v_0\|^2$, and hence $\alpha_j = 0$.

  If $v \in L$, then there are unique $w \in L_W, \beta \in \Z$ such that $w + \beta v_0$. Since $\langle v_0, w\rangle = 0$, we see that
  \[
    \|v\|^2 = \|w\|^2 + \beta^2 \|v_0\|^2.
  \]
  In particular, if $\|v\| \leq \|v_0\|$, then $|\beta| \leq 1$. In the case that $\beta = \pm 1$, we have $w = 0$ and so $v = \pm v_0$. If $\beta = 0$, then $v \in W$. Applying this reasoning to the $b_i$ for $i \leq k$, we see that either all of the $b_i \in W$, or for some $i$, $b_i = \pm v_0$. Since $v_0 \notin W$, the first possibility cannot hold. 
\end{proof}

\begin{proposition}\label{prop:unique-orthogonal-basis}
  Let $\Lambda$ be an orthogonal lattice. Then, up to sign and rearrangement, there is a unique choice of orthogonal basis for $\Lambda$.
\end{proposition}

\begin{proof}
  Let $b_1, \dots, b_m$ be an orthogonal basis for $\Lambda$. Then it suffices to show that for any other orthogonal basis $d_1, \dots, d_m$ and for every $k$, we have $b_k = \pm d_j$ for some $j$. Let $V$ be the $\R$-span of $b_1, \dots, b_m$, $v_0 = b_k$, and $W$ the $\R$-span of $b_1, \dots, b_{k-1}, b_{k+1}, \dots, b_m$. Let $L_W = \displaystyle\sum_{i \neq k} \Z b_i$ and $L = L_W + \Z v_0$. The hypotheses of Prop.~\ref{prop:one-vector-orthogonal} are satisfied, and so $b_k = \pm d_j$ for some $j$, as required.
\end{proof}

\begin{proposition}\label{prop:orthog-lattice-lambda_i}
Let $\Lambda(B)$ be a lattice with orthogonal basis $B=\{b_1,\ldots,b_m\}$ satisfying $\|{b_i}\|\leq \|{b_{i+1}}\|$ for $i = 1, \dots, m-1$. Then for each $i$, $\lambda_i(\Lambda) = \norm{b_i}$.
\end{proposition}

\begin{proof}
  We use induction on $i$. The $i = 1$ case is just Lemma~\ref{lemma:lambda1-b1}. Now suppose $\lambda_j = \|b_j\|$ for $j = 1, \dots, i$. If $\|b_{i+1}\| = \|b_{i}\|$, then the closed ball $B_{\lambda_{i}}$ contains $b_1, \dots, b_{i+1}$, and therefore
  \[
    \dim (\Lambda_\R \cap B_{\lambda_i}) \geq i+1.
  \]
  In particular, this means $\lambda_{i+1} \leq \lambda_i$. But \emph{a priori} $\lambda_{i+1} \geq \lambda_i$. Therefore $\lambda_{i+1} = \lambda_i = \|b_i\| = \|b_{i+1}\|$, as required.

  Now suppose that $\|b_{i+1}\| > \|b_i\|$. Given $v \in \Lambda$, we have $v = \sum \alpha_j b_j$ for some $\alpha_j \in \Z$, so $\|v\|^2 = \sum \alpha_j^2 \|b_j\|^2$. If $\|v\| < \|b_{i+1}\|$, then $\alpha_j = 0$ for all $j \geq i+1$. Thus for $r < \|b_{i+1}\|$, we have
  \[
    \dim (\Lambda_\R \cap B_{r}) \leq i.
  \]
  But if $r = \|b_{i+1}\|$, then $b_1, b_2, \dots, b_{i+1} \in B_r$, and hence $\dim (\Lambda_\R \cap B_r) \geq i+1$. Therefore $\lambda_{i+1} = \|b_{i+1}\|$.
\end{proof}

\begin{lemma}\label{lemma:rank2LatticeNotOrthogonal}
Let $\Lambda$ be a rank two lattice. Suppose $w_1, w_2 \in \Lambda$ are linearly independent with $\lambda_1=\norm{w_1}$, $\lambda_2 = \norm{w_2}$, and $w_1$ not orthogonal to $w_2$. Then $\Lambda$ is not orthogonal.
\end{lemma}
\begin{proof}
Suppose to the contrary that $\Lambda$ is orthogonal. Then there is a basis $B=\{b_1,b_2\}$ of $\Lambda$ with $b_1 \perp b_2$. Assume $\norm{b_1} \leq \norm{b_{2}}$. Then by Proposition~\ref{prop:orthog-lattice-lambda_i}, $\norm{b_1} = \norm{w_1}$ and $\norm{b_2}=\norm{w_2}$. Since $w_1,w_2 \in \Lambda$, we know 
\[w_1 = a_1b_1+a_2b_2 \quad \text{and} \quad w_2 = c_1b_1+c_2b_2\]
for some $a_i,c_i \in \Z$. So
\[\norm{w_1}^2=a_1^2\norm{b_1}^2 + a_2^2\norm{b_2}^2 \quad \text {and} \quad \norm{w_2}^2=c_1^2\norm{b_1}^2 + c_2^2\norm{b_2}^2\]
since $B$ is orthogonal. Suppose $\lambda_1 < \lambda_2$. Then $\norm{b_1}^2= \norm{w_1}^2$ implies $a_1=\pm 1$ and $a_2 = 0$.
Similarly, $\norm{b_2}^2=\norm{w_2}^2$ implies $c_2=\pm 1$ and $c_1 =0$. So $w_1 = \pm b_1$ and $w_2 = \pm b_2$. But this implies $w_1 \perp w_2$, which is a contradiction.

Now suppose $\lambda_1 = \lambda_2$. Then $\norm{b_2}^2 = \norm{b_1}^2 = \norm{w_1}^2$ implies $a_1=\pm 1$ and $a_2 = 0$ or $c_2=\pm 1$ and $c_1 =0$. So $w_1 = \pm b_1$ or $w_1 = \pm b_2$. Similarly, $\norm{b_1}^2 = \norm{b_2}^2=\norm{w_2}^2$ implies $c_2=\pm 1$ and $c_1 =0$ or $a_1=\pm 1$ and $a_2 = 0$. So $w_2 = \pm b_2$ or $w_2 = \pm b_1$. Hence either $w_1 \perp w_2$ or $w_1 = \pm w_2$. But either case yields a contradiction.
\end{proof}

\subsection{Biquadratic number fields}
\label{sec:biqu-numb-fields}

Recall that a \emph{number field} is a finite extension of $\Q$. A \emph{biquadratic} number field is one of the form $K = \Q(\sqrt{d_1}, \sqrt{d_2})$, $d_1,d_2 \in \Q$, for which none of $d_1, d_2, d_1d_2$ are squares in $\Q$. Such a number field is Galois over $\Q$ with Galois group $G \cong \Z/2\Z \times \Z/2\Z$. We define generators $\sigma, \tau$ for $G$ by
\begin{align*}
  \sigma\left(\sqrt{d_1}\right) &= -\sqrt{d_1}, \\ \sigma\left(\sqrt{d_2}\right) &= \sqrt{d_2}, \\
  \tau\left(\sqrt{d_1}\right) &= \sqrt{d_1}, \textrm{ and} \\ \tau\left(\sqrt{d_2}\right) &= -\sqrt{d_2}.
\end{align*}
We further assume that $d_1, d_2 > 0$, so that $K$ is totally real.

Fix an embedding of $K$ in $\R$. Define $\Log: K^\times \to \R^4$ by
\[
  \Log(x) = (\log |x|, \log |\tau(x)|, \log |\sigma(x)|, \log |\sigma\tau(x)|)
\]
where the absolute value is that from $\R$. The kernel of $\Log$ consists of the roots of unity in $K$. Since $K$ is real, the only roots of unity are $\pm 1$.

Let $\ok$ be the {ring of integers} of $K$. 
By Dirichlet's Unit Theorem~\cite[Theorem 7.3]{neukirch}, $\Log(\okx)$ is a rank $3$ lattice in $\R^4$ which is contained in the hyperplane $x_1 + x_2 + x_3 + x_4 = 0$. We call this lattice the \emph{log-unit lattice} attached to $K$. We will use the notation $\Lambda_K$ to refer to this lattice. For a general number field $K$, the log-unit lattice $\Lambda_K = \Log(\okx)$ can be defined similarly; for example, if $K$ is a real quadratic field, then $\Lambda_K$ is a rank $1$ lattice in $\R^2$. We say $\epsilon_1, \cdots, \epsilon_r \in K$ is a \emph{system of fundamental units for $K$} if $\Log(\epsilon_1), \dots, \Log(\epsilon_r)$ forms a basis for $\Lambda_K$. If $r=1$, we say $\epsilon_1$ is a fundamental unit for $K$.

Given $d_1, d_2$, let $K_1, K_2, K_3$ be $\Q(\sqrt{d_1}), \Q(\sqrt{d_2})$, and $\Q(\sqrt{d_1d_2})$ respectively. These are the three quadratic subfields of $K$, each of which is real. By Dirichlet's Unit Theorem, the associated log-unit lattice for each $K_i$ has rank $1$. Let $\epsilon_i \in K_i$ be a {fundamental unit} for $K_i$. 
The following results are due to~\cite{kubota1956uber}, but we cite \cite{mazur2004biquadratic} for convenience. As in \cite[p. 106]{mazur2004biquadratic}, we categorize real biquadratic fields into four different types, characterized by what constitutes a system of fundamental units. These types are as follows---in each case, we specify a system of fundamental units for $K$.
\begin{description}
    \item[Type I] Up to permutation of subscripts, one of
  \begin{enumerate}[(a)]
      \item $\epsilon_1, \epsilon_2, \epsilon_3$;
      \item $\sqrt{\epsilon_1}, \epsilon_2, \epsilon_3$; or
      \item $\sqrt{\epsilon_1}, \sqrt{\epsilon_2}, \epsilon_3$.
\end{enumerate}

    \item[Type II] Up to permutation of subscripts, one of
  \begin{enumerate}[(a)]
      \item $\sqrt{\epsilon_1\epsilon_2}, \epsilon_2, \epsilon_3$ or 
      \item $\sqrt{\epsilon_1\epsilon_2}, \epsilon_2, \sqrt{\epsilon_3}$.
  \end{enumerate}
    \item[Type III] $\sqrt{\epsilon_1\epsilon_2}, \sqrt{\epsilon_2\epsilon_3}, \sqrt{\epsilon_1\epsilon_3}$.
    \item[Type IV] $\sqrt{\epsilon_1\epsilon_2\epsilon_3}, \epsilon_2, \epsilon_3$.
\end{description}
Thus, for example, if $\epsilon_1, \epsilon_2, \sqrt{\epsilon_3}$ forms a system of fundamental units for $K$, then $K$ is of Type I. It is easy to check that the types are disjoint.

Kubota gave explicit families of biquadratic number fields in 3 of these types. 
\begin{theorem}[Theorem 2 of \cite{mazur2004biquadratic}] \label{thm:mazur-ullom-main}
  Let $p_1, p_2$ be distinct primes, $K = \Q(\sqrt{p_1}, \sqrt{p_2})$, and $\epsilon_1, \epsilon_2, \epsilon_3$ fundamental units for $\Q(\sqrt{p_1})$, $\Q(\sqrt{p_2})$, $\Q(\sqrt{p_1p_2})$, respectively. Let $e_i = \Nm_{K_i/\Q}(\epsilon_i)$.
  \begin{enumerate}
      \item Suppose $p_1 \equiv 1 \pmod{4}$, $p_2 \not\equiv 3 \pmod{4}$, and $e_1 = e_2 = -1$. If $e_3 = 1$, then $K$ is of type I, and $\epsilon_1, \epsilon_2, \sqrt{\epsilon_3}$ is a system of fundamental units. If $e_3 = -1$, then $K$ is of type IV.
      \item Suppose $p_1 \equiv 3 \pmod{4}$ and $p_2 = 2$. Then $K$ is of type I, and $\sqrt{\epsilon_1}, \epsilon_2, \sqrt{\epsilon_3}$ is a system of fundamental units for $K$.
      \item Suppose $p_1 \equiv p_2 \equiv 3 \pmod{4}$. Then $K$ is of type II. 
  \end{enumerate}
\end{theorem}
See \cite{mazur2004biquadratic} for additional related results, including a proof that there are infinitely main fields of all 4 types.

Our main theorem is the following:
\begin{theorem}\label{thm:all-types}
  Suppose $K$ is a real biquadratic field. The log-unit lattice of $K$ is orthogonal if and only if $K$ is of type I.
\end{theorem}
Combining the above theorem and Theorem~\ref{thm:mazur-ullom-main}, we immediately obtain Theorems~\ref{thm:typeI-primes}, \ref{thm:typeII-primes}, and \ref{thm:typeiv-primes}.

Observe that 
\begin{align*}
  \tau(\epsilon_i) &=
          \begin{cases}
            \pm\epsilon_i^{-1} & i \neq 1 \\
            \pm\epsilon_1 & i = 1
          \end{cases} \\
  \sigma(\epsilon_i) &=
          \begin{cases}
            \pm\epsilon_i^{-1} & i \neq 2 \\
            \pm\epsilon_2 & i = 2
          \end{cases} \\
  (\sigma\tau)(\epsilon_i) &=
           \begin{cases}
             \pm\epsilon_i^{-1} & i \neq 3 \\
             \pm\epsilon_3 & i = 3.
           \end{cases}
\end{align*}
Let ${w}_i=\Log(\epsilon_i)$. Then
\begin{align*}
{w}_1 &= \left(\log|\epsilon_1|,\log|\epsilon_1|,-\log|\epsilon_1|,-\log|\epsilon_1|\right)
\\{w}_2 &= \left(\log|\epsilon_2|,-\log|\epsilon_2|,\log|\epsilon_2|,-\log|\epsilon_2|\right)
\\{w}_3 &= \left(\log|\epsilon_3|,-\log|\epsilon_3|,-\log|\epsilon_3|,\log|\epsilon_3|\right).
\end{align*}
\begin{lemma}\label{lemma:wi-orthogonal-lengths}
  Suppose $1 \leq i,j \leq 3$, $i \neq j$. Then
  \begin{enumerate}[(a)]
      \item $\braket{w_i,w_j} = 0$,
      \item $\frac{\log |\epsilon_i|}{\log |\epsilon_j|}$ is not algebraic over $\Q$, and
      \item $\norm{w_i} \neq \norm{w_j}$.
  \end{enumerate}
\end{lemma}

\begin{proof}
  The first claim is straightforward.

  For the second claim, if $\frac{\log |\epsilon_i|}{\log |\epsilon_j|} \in \Q$, then $\exists a, b \in \Z$, not both zero, such that
  \[
    a\log |\epsilon_i| = b\log |\epsilon_j|.
  \]
  Thus $\log |\epsilon_i^a\epsilon_j^{-b}| = 0$, and so $\epsilon_i^a\epsilon_j^{-b}$ is a root of unity in $K$. Since $K$ is totally real, $\epsilon_i^a\epsilon_j^{-b} = \pm 1$. But $\epsilon_i, \epsilon_j$ are units of infinite order in the linearly disjoint fields $K_i,K_j$, respectively. Thus $a = b = 0$. This contradicts our choice of $a,b$, and hence $\frac{\log |\epsilon_i|}{\log |\epsilon_j|} \notin \Q$. The claim then follows from linear independence of logarithms~\cite{baker1966}.

  Since $\norm{w_i} = 2|\log |\epsilon_i||$, the last claim follows.
\end{proof}

Theorem~\ref{thm:all-types} follows from considering each type individually. To be precise, we will show that log-unit lattices for type I fields are orthogonal in Theorem~\ref{thm:typeI-orthogonal}, and that log-unit lattices for types II, III, and IV are not orthogonal in Theorems~\ref{thm:typeII-not-orthogonal}, \ref{thm:typeIII-not-orthogonal}, and \ref{thm:typeIV-not-orthogonal}, respectively. The key fact we will be using is that $\Lambda(w_1, w_2, w_3)$ is an orthogonal sublattice of the log-unit lattice $\Lambda_K$. Through \S{}6, we will assume that $K$ is a real biquadratic field given by $\Q(\sqrt{d_1}, \sqrt{d_2})$, and that the $\epsilon_i, w_i$ are as above.

\section{Type I lattices}
\label{sec:type-i-lattices}

\begin{theorem}\label{thm:typeI-orthogonal}
  Suppose $K$ is a real biquadratic field of type I. Then $\Lambda_K$ is orthogonal.
\end{theorem}

\begin{proof}
  Recall the vectors $w_1, w_2, w_3$, which are pairwise orthogonal. If $K$ is of type Ia, then $\Lambda_K = \Lambda(w_1, w_2, w_3)$. If $K$ is of type Ib, then up to renumbering, $\Lambda_K = \Lambda(w_1/2, w_2, w_3)$. If $K$ is of type Ic, then up to renumbering, $\Lambda_K = \Lambda(w_1/2, w_2/2, w_3)$. In all three cases, the given basis is orthogonal.  
%
\end{proof}

\begin{proof}[Proof of Corollary~\ref{cor:typeI-covering-radius-bounds}]
  Let $R_i = \ln\abs{\varepsilon_i}$ for all $i$. By replacing $\epsilon_i$ with its inverse if necessary, we may assume that $R_i > 0$.

  If we assume hypothesis $(1)$, then by Theorem~\ref{thm:mazur-ullom-main} we may set $v_1 = w_1$, $v_2 = w_2$, and $v_3 = \frac{1}{2}w_3$. Then $\Lambda_K = \Lambda(v_1,v_2,v_3)$. As the $w_i$ are orthogonal, so are the $v_i$. 

  If we assume hypothesis $(2)$, then by Theorem~\ref{thm:mazur-ullom-main}, we may set $v_1 = \frac{1}{2}w_1$, $v_2 = w_2$, and $v_3 = \frac{1}{2}w_3$. Then $\Lambda_K = \Lambda(v_1,v_2,v_3)$. The $v_i$ are orthogonal in this case as well.

By~\cite[p. 212]{MR1387478}, $\tilde{\Delta}_i \leq R_i$. By \cite[p. 329]{hua1982nt}, $R_i < \hat{\Delta}_i$. Since $\norm{w_i} = 2 R_i$, the claim follows.
\end{proof}

\section{Type II lattices}
\label{sec:type-ii-lattices}

Suppose $K$ is of type II. We will assume that either $\sqrt{\epsilon_1\epsilon_2}, \epsilon_2, \epsilon_3$ or $\sqrt{\epsilon_1\epsilon_2}, \epsilon_2, \sqrt{\epsilon_3}$ forms a system of fundamental units; the cases where we permute the subscripts will follow by an identical argument. Since in cases IIa and IIb, $\sqrt{\epsilon_1\epsilon_2}, \epsilon_1, \epsilon_3$ and $\sqrt{\epsilon_1\epsilon_2}, \epsilon_1, \sqrt{\epsilon_3}$ respectively also forms a system of fundamental units, we may assume that $\|w_1\| < \|w_2\|$. Let
\begin{align*}
  v_1 &= \frac{1}{2}(w_1 + w_2) \textrm{ and}\\
  \tilde{v}_2 &= w_2.
\end{align*}
If $K$ is of type IIa, let $v_3 = w_3$; otherwise let $v_3 = \frac{1}{2}w_3$. We have $\Lambda_K = \Lambda(v_1, \tilde{v}_2, v_3)$. It will be convenient to set $v_2 = \frac{1}{2}(w_2 - w_1)$; then $\Lambda_K = \Lambda(v_1, v_2, v_3)$.


\begin{lemma}\label{lemma:typeii-sublattice}
The sublattice $\Lambda(v_1,v_2) \subseteq \Lambda_K$ is not orthogonal.
\end{lemma}

\begin{proof}
Let $\Lambda = \Lambda(v_1,v_2) \subseteq \Lambda_K$. Note that
\begin{align*}
w_1 &= v_1 - v_2 
\textrm{ and} \\w_2 &= v_1 + v_2. 
\end{align*}
Since $w_1 \perp w_2$ and
$v_1=\frac{1}{2}(w_1+w_2)$, we have
\[\norm{v_1}^2 = \frac{1}{4}\norm{w_1}^2 + \frac{1}{4}\norm{w_2}^2.\]
Recall that $\norm{w_1} < \norm{w_2}$, so
\[\frac{1}{4}\norm{w_1}^2 + \frac{1}{4}\norm{w_2}^2 < \frac{1}{4}\norm{w_2}^2 + \frac{1}{4}\norm{w_2}^2 < \norm{w_2}^2.\]
Thus $\norm{v_1} < \norm{w_2}$.

Now let $x \in \Lambda \setminus \{0\}$. Then $x = a_1v_1+a_2v_2$ for some $a_i \in \Z$ with at least one nonzero $a_i$. Observe that
\[x = a_1\left(\frac{1}{2}(w_1+w_2)\right) + a_2\left(\frac{1}{2}(-w_1+w_2)\right) = \frac{1}{2}(a_1-a_2)w_1 + \frac{1}{2}(a_1+a_2)w_2.\]
Thus
\[\norm{x}^2 = \frac{1}{4}(a_1-a_2)^2\norm{w_1}^2 + \frac{1}{4}(a_1+a_2)^2\norm{w_2}^2\]
since $w_1 \perp w_2$. Let $u_1= a_1-a_2$ and $u_2 = a_1+a_2$. To minimize $\norm{x}$, we would like to minimize $\abs{u_1}$ and  $\abs{u_2}$, where $u_1, u_2 \in \Z$. 
As $u_1 - u_2 = -2a_2$ and $a_2\in \Z$, we have $u_1\equiv u_2 \mod 2$. To minimize $\|x\|$, we must have one of the following: $u_1 = u_2 = \pm 1$; or $u_1 = - u_2 = \pm 1$; or $u_1 = \pm 2, u_2 = 0$; or $u_1 = 0, u_2 = \pm 2$. In other words, $x \in \{\pm v_1, \pm v_2, \pm w_1, \pm w_2\}$. Observe that $\norm{v_1} = \norm{v_2}$, $\norm{w_2} > \norm{w_1},$ and $\norm{w_2} > \norm{v_1}$, so $\lambda_1(\Lambda) = \norm{b_1}$ and $\lambda_2(\Lambda) = \norm{b_2}$ for some $b_1,b_2 \in \{v_1, v_2, w_1\}$. But $\braket{v_i,w_1} = \pm \frac{1}{2}\norm{w_1}^2 \neq 0$ for $i = 1,2$, and $\braket{v_1, v_2} = \frac{1}{4}(\norm{w_2}^2 - \norm{w_1}^2) \neq 0$. Therefore $b_1$ cannot be orthogonal to $b_2$. By Lemma~\ref{lemma:rank2LatticeNotOrthogonal}, $\Lambda$ is not orthogonal.
\end{proof}

\begin{theorem}\label{thm:typeII-not-orthogonal}
  Suppose $K$ is a real biquadratic field of type II. Then $\Lambda_K$ is not orthogonal.
\end{theorem}

\begin{proof}
  Suppose $(b_1, b_2, b_3)$ is an orthogonal basis for $\Lambda_K$. Let $v_1, v_2, v_3$ be as above. Let $W$ be the $\R$-linear span of $v_1$ and $v_2$. Note that $v_3$ is orthogonal to $W$. By Proposition~\ref{prop:one-vector-orthogonal}, one of the $b_i$---say, $b_3$---equals $\pm v_3$. But then $(b_1, b_2)$ must form an orthogonal basis for $\Lambda(v_1,v_2)$, which contradicts Lemma~\ref{lemma:typeii-sublattice}.
\end{proof}

\section{Type III lattices}
\label{sec:type-iii-lattices}

Let $K$ be a type III real biquadratic number field. Define new vectors $\hat{w}_1, \hat{w}_2, \hat{w}_3$ as being the ${w}_i$ in increasing order by length; that is, $\{w_1,w_2,w_3\} = \{\hat{w}_1,\hat{w}_2,\hat{w}_3\}$, and $\norm{\hat{w}_1} < \norm{\hat{w}_2} < \norm{\hat{w}_3}$.
Let
\begin{align*}
  v_1 &= \frac{1}{2}\left(\hat{w}_2+\hat{w}_3\right), \\
  v_2 &= \frac{1}{2}\left(\hat{w}_1+\hat{w}_3\right), \textrm{ and} \\
  v_3 &= \frac{1}{2}\left(\hat{w}_1+\hat{w}_2\right).
\end{align*}
 Let $B = \{v_1, v_2, v_3\}$. By definition of type III, $B$ forms a basis for $\Lambda_K$. We have $\left<v_1,v_2\right> = \frac{1}{4} \norm{\hat{w}_3}^2$, $\left<v_1,v_3\right> = \frac{1}{4} \norm{\hat{w}_2}^2,$ and $\left<v_2,v_3\right> = \frac{1}{4} \norm{\hat{w}_1}^2$ since the $\hat{w}_i$ are pairwise orthogonal. Hence $B$ is not orthogonal. Our goal is to show that there is no orthogonal basis for $\Lambda_K$. 
 
Let $x \in \Lambda_K$. Then $x = \sum_{i=1}^3 a_iv_i$ for $a_i \in \Z$. So
\begin{align*}
x &= a_1\left(\frac{1}{2}(\hat{w}_2+\hat{w}_3)\right)+ a_2\left(\frac{1}{2}(\hat{w}_1+\hat{w}_3)\right) + a_3\left(\frac{1}{2}(\hat{w}_1+\hat{w}_2) \right)
\\&= \left(\frac{a_2+a_3}{2}\right)\hat{w}_1 +\left(\frac{a_1+a_3}{2}\right)\hat{w}_2 + \left(\frac{a_1+a_2}{2}\right)\hat{w}_3.
\end{align*}
Write $u_1 = a_2 + a_3$, $u_2 = a_1 + a_3$, and $u_3 = a_1 + a_2$. Then 
\begin{equation}
x = \sum_{i=1}^3 \frac{u_i}{2}\hat{w}_i. \label{eq:1}
\end{equation}

\begin{lemma}\label{lemma:typeIIIC=H}
Let
\[C = \left\{(u_1,u_2,u_3) \in \Z^3 \mid \sum_{i=1}^3 u_i \equiv 0 \mod 2\right\}.\]
Then $C$ is generated by $(2,0,0),(0,2,0),(0,0,2),(1,1,0),(1,0,1),$ and $(0,1,1)$.
\end{lemma}

\begin{proof}
  This is an easy exercise.
\end{proof}

\begin{lemma}\label{lemma:typeIIIC=S}
Let $C$ be as above. Then
\[C = \left\{(u_1,u_2,u_3) \in \Z^3 \mid \sum_{i=1}^3 \frac{u_i}{2}\hat{w}_i \in \Lambda_K\right\}.\]
\end{lemma}

\begin{proof}
Write $S$ for the right-hand set above, and let $u = (u_1,u_2,u_3) \in S$. Recall that $u_1 = a_2+a_3$, $u_2 = a_1+a_3$, and $u_3 = a_1+a_2$ for some $a_i \in \Z$. So
\[\sum_{i=1}^3 u_i = 2(a_1+a_2+a_3) \equiv 0 \mod 2.\]
Thus, $u \in C$. To show that $C \subseteq S$ we simply show that the generators of $C$, as in Lemma~\ref{lemma:typeIIIC=H}, are in $S$. 
Note that $0\hat{w}_1+\frac{1}{2}\hat{w}_2+\frac{1}{2}\hat{w}_3 = v_1, \frac{1}{2}\hat{w}_1+0\hat{w}_2+\frac{1}{2}\hat{w}_3 = v_2$, and $\frac{1}{2}\hat{w}_1+\frac{1}{2}\hat{w}_2+0\hat{w}_3 = v_3$ are clearly in  $\Lambda_K(v_1,v_2,v_3)$. So $(0,1,1),(1,0,1),(1,1,0) \in S$. By a similar calculation, $(2,0,0),(0,2,0),(0,0,2)$ correspond to $\hat{w}_1,\hat{w}_2,$ and $\hat{w}_3$, respectively.
\end{proof}

\begin{lemma}\label{lemma:typeIII-pairwise-lin-ind}
Let $\mathscr{S}=\set{v_1,  v_2,  v_3,  v_1',   v_2',  v_3',  \hat{w}_1,  \hat{w}_2}$, 
where 
    \[v_1' = \frac{1}{2}\left(\hat{w}_2-\hat{w}_3\right), \quad v_2' = \frac{1}{2}\left(\hat{w}_1-\hat{w}_3\right), \quad \textnormal{and} \quad v_3' =\frac{1}{2}\left(\hat{w}_1-\hat{w}_2\right).\]
 
    Let $x,y \in \mathscr{S}$ with $x\neq y$. If $\alpha_1 x + \alpha_2 y=0$, then $\alpha_1 = \alpha_2=0$.
\end{lemma}
\begin{proof}
  One checks each case. Alternatively, each element of $\mathscr{S}$ can be written in terms of the $\hat{w}_i$, and from these representations it is clear that no two are parallel.
\end{proof}

\begin{lemma}\label{lemma:typeIII-shortest-vectors-not-orthogonal}
There exist
\[b_1, b_2 \in \mathscr{S}=\left\{ v_1,  v_2,  v_3,  v_1',   v_2',  v_3',  \hat{w}_1,  \hat{w}_2\right\}\]
such that $\lambda_1(\Lambda_K)=\norm{b_1}$ and $\lambda_2(\Lambda_K) = \norm{b_2}$. 
\end{lemma}

\begin{proof}
Let $x \in \Lambda_K \setminus \{0\}$, and let $a_i, u_i$ be as in~\eqref{eq:1}.  Then $\norm{x}^2 = \frac{1}{4}\sum_{i=1}^3 u_i^2 \norm{\hat{w}_i}^2$ since $\hat{w}_i \perp \hat{w}_j$ for all $i \neq j$. To minimize $\norm{x}$, we must find $(u_1, u_2, u_3) \in C$ whose coordinates have minimal absolute value. It follows from Lemma~\ref{lemma:typeIIIC=H} and Lemma~\ref{lemma:typeIIIC=S} that $\norm{x}$ is minimized when $(\abs{u_1},\abs{u_2}, \abs{u_3})$ is in the list
\[(1,1,0),\ (1,0,1),\ (0,1,1),\ (2,0,0),\ (0,2,0),\ \text{and }(0,0,2).\]
In particular, $\left(\abs{u_1}, \abs{u_2}, \abs{u_3}\right) = (1,1,0)$ gives $\pm v_3$, $\pm v_3'$, $\left(\abs{u_1}, \abs{u_2}, \abs{u_3}\right) = (1,0,1)$ gives $\pm v_2$, $\pm v_2'$, and $\left(\abs{u_1}, \abs{u_2}, \abs{u_3}\right) = (0,1,1)$ gives $\pm v_1$, $\pm v_1'$. Similarly, $(2,0,0)$, $(0,2,0)$, $(0,0,2)$ give $\pm \hat{w}_1$, $\pm \hat{w}_2$, and $\pm \hat{w}_3$, respectively. Since $\hat{w}_i \perp \hat{w}_j$ for $i \neq j$, $\norm{v_1}^2 = \frac{1}{4}\left(\norm{\hat{w}_2}^2 + \norm{\hat{w}_3}^2\right)$. Furthermore, $\norm{\hat{w}_2} < \norm{\hat{w}_3}$, so
\[\frac{1}{4}\left(\norm{\hat{w}_2}^2 + \norm{\hat{w}_3}^2\right) < \frac{1}{4}\left(\norm{\hat{w}_3}^2 + \norm{\hat{w}_3}^2\right) < \norm{\hat{w}_3}^2.\]
Thus $\norm{v_1} < \norm{\hat{w}_3}$. Note that $\norm{v_i} = \norm{v_i'}$ for all $i$, and $\norm{v_3} < \norm{v_2} < \norm{v_1}$. In particular $\norm{\hat{w}_3}$ is larger than the norm of every element of $\mathscr{S}$. 
    By Lemma \ref{lemma:typeIII-pairwise-lin-ind}, all pairs $b,b' \in \mathscr{S} $ with $b \neq b'$ are linearly independent. The result follows.
\end{proof}

\begin{theorem}\label{thm:typeIII-not-orthogonal}
Let $K$ be a type III real biquadratic number field.
Then $\Lambda_K$ is not orthogonal.
\end{theorem}

\begin{proof}
    Suppose otherwise. Then $\Lambda_K$ has an orthogonal basis, say, $B=\set{b_1,b_2,b_3}$. Without loss of generality, $\norm{b_1} \leq \norm{b_2} \leq \norm{b_3}$. Then by Proposition~\ref{prop:orthog-lattice-lambda_i}, $\lambda_i(\Lambda_K)=\norm{b_i}$ for $i = 1,2,3$, and by Lemma~\ref{lemma:typeIII-shortest-vectors-not-orthogonal},
    \[\pm b_1, \pm b_2 \in \mathscr{S} = \set{v_1,v_2,v_3,v_1',v_2',v_3',\hat{w}_1,\hat{w}_2},\]
    where the $v_i'$ are as in Lemma~\ref{lemma:typeIII-pairwise-lin-ind}. But the only pairs of orthogonal vectors in $\mathscr{S}$ are
    \begin{align*}
      \hat{w}_1, \hat{w_2};& \\
      \hat{w}_1, v_1;& \, \hat{w_1}, v_1'; \\
      \hat{w}_2, v_2;& \textrm{ and } \hat{w_2}, v_2'.
    \end{align*}
    But in each case, $v_3$ is shorter than the second element of each pair, and so none of these pairs can equal $\pm b_1, \pm b_2$. Therefore no such $B$ can exist, proving the claim.
\end{proof}

\section{Type IV lattices}
\label{sec:type-iv-lattices}

Suppose $K$ is of type IV. As before, define $\hat{w}_1, \hat{w}_2, \hat{w}_3$ as vectors satisfying $\{w_1,w_2,w_3\} = \{\hat{w}_1,\hat{w}_2,\hat{w}_3\}$, and $\norm{\hat{w}_1} < \norm{\hat{w}_2} < \norm{\hat{w}_3}$. Let $v_1 = \frac{1}{2}\sum_{i=1}^3 \hat{w}_i$, $v_2 = \hat{w}_2$, and $v_3 = \hat{w}_3$. Let $B = \{v_1, v_2, v_3\}$, so that $B$ forms a basis for $\Lambda_K$. We have that $\braket{v_1,v_2}=\frac{1}{2}\norm{\hat{w}_2}^2$, $\braket{v_1,v_3}=\frac{1}{2}\norm{\hat{w}_3}^2$, and $\braket{v_2,v_3}=0$ since the $\hat{w}_i$ are pairwise orthogonal. Hence $B$ is not orthogonal.
 
Let $x \in \Lambda_K$. Then $x = \sum_{i=1}^3 a_iv_i$ for $a_i \in \Z$. So
\begin{align}\label{eq:2}
    x &= a_1\left(\frac{1}{2}\sum_{i=1}^3 \hat{w}_i\right)+ a_2\hat{w}_2 + a_3\hat{w}_3 \notag \\ 
      &= \frac{a_1}{2}\hat{w}_1 + \left(\frac{a_1+2a_2}{2}\right)\hat{w}_2 + \left(\frac{a_1+2a_3}{2}\right)\hat{w}_3.
\end{align}
Write $u_1 = a_1$, $u_2 = a_1 + 2a_3$, and $u_3 = a_1 + 2a_2$. 

\begin{lemma}\label{lemma:typeIVC=H}
Let
\[C = \{(u_1,u_2,u_3) \in \Z^3 \mid u_1 \equiv u_2 \equiv u_3 \mod 2\}.\]
Then $C$ is generated by $(2,0,0),(0,2,0),(0,0,2)$ and $(1,1,1)$.
\end{lemma}

\begin{proof}
Let $H < \Z^3$ be generated by $(2,0,0),(0,2,0),(0,0,2)$ and $(1,1,1)$. Certainly $H \subset C$. Let $(u_1, u_2, u_3) \in C$, so $u_i \in \Z$ with $u_1 \equiv u_2 \equiv u_3 \mod 2$. If each $u_i \equiv 0 \mod 2,$ then $(u_1, u_2, u_3) \in 2\Z^3 \subseteq H$. If each $u_i \equiv 1 \mod 2$, then $u_i = 2n_i + 1$ for some $n_i \in \Z$. Then
\[(u_1, u_2, u_3) = (2n_1, 2n_2, 2n_3) + (1, 1, 1)\]
which lies in $H$.
\end{proof}

\begin{lemma}\label{lemma:typeIVC=S}
Let $C$ be as above. Then
\[C = \{(u_1,u_2,u_3) \in \Z^3 \mid \sum_{i=1}^3 \frac{u_i}{2}\hat{w}_i \in \Lambda_K\}.\]
\end{lemma}

\begin{proof}
  Let $S$ be the right-hand set in the statement, and let $x \in S$. Then $x  = (u_1, u_2, u_3)$ for some $u_i \in \Z$ with $\sum_{i=1}^3 \frac{u_i}{2}\hat{w}_i \in \Lambda_K$. Recall that $u_1 = a_1,\ u_2 = a_1 + 2a_3$ and $u_3 = a_1 + 2a_2$ for some $a_i \in \Z$. If $a_1 \equiv 0 \mod 2$, then each $u_i \equiv 0 \mod 2$. If $a_1 \equiv 1 \mod 2$, then each $u_i \equiv 1 \mod 2$. In either case, $x \in C$. For the other inclusion, it suffices to show that the generators of $C$ are in $S$; this is a routine calculation. 
\end{proof}

\begin{lemma}\label{lemma:typeIV-pairwise-lin-ind}
    Let $\mathscr{S}=\set{v_1, v_1', v_1'', v_1''',  \hat{w}_1, \hat{w}_2}$, where   
    \begin{align*}
    v_1' &= v_1 - \hat{w}_1 \\ &= - \frac{1}{2}\hat{w}_1 + \frac{1}{2}\hat{w}_2 + \frac{1}{2}\hat{w}_3,\\
v_1'' &=  v_1 - \hat{w}_2 \\ &= \frac{1}{2}\hat{w}_1 - \frac{1}{2}\hat{w}_2 + \frac{1}{2}\hat{w}_3, \textnormal{ and}
\\
v_1''' &= v_1 - \hat{w}_3 \\ &= \frac{1}{2}\hat{w}_1 + \frac{1}{2}\hat{w}_2 - \frac{1}{2}\hat{w}_3.
    \end{align*}
    Let $x,y \in \mathscr{S}$ with $x\neq y$. Then $\alpha_1 x + \alpha_2 y=0$ implies $\alpha_1,\alpha_2=0$.
\end{lemma}

\begin{proof}
Suppose $x = v_1$ and $y = v_1'$. Observe that
\begin{align*}
0 &= \alpha_1 v_1 + \alpha_2 v_1'
\\&= \alpha_1 \left(\frac{1}{2} \left(\hat{w}_1 + \hat{w}_2 + \hat{w}_3\right)\right) + \alpha_2 \left(\frac{1}{2} \left(-\hat{w}_1 + \hat{w}_2 + \hat{w}_3\right)\right)
\\&= \frac{\left(\alpha_1 - \alpha_2\right)}{2} \hat{w}_1 + \frac{\left(\alpha_1 + \alpha_2\right)}{2} \hat{w}_2 + \frac{\left(\alpha_1 + \alpha_2\right)}{2} \hat{w}_3
\end{align*}
Multiplying by $2$ yields
\[0 =\left(\alpha_1 - \alpha_2\right) \hat{w}_1 + \left(\alpha_1 + \alpha_2\right) \hat{w}_2 + \left(\alpha_1 + \alpha_2\right) \hat{w}_3.\]
But $\{\hat{w}_1, \hat{w}_2, \hat{w}_3\}$ is a linearly independent set, so $\alpha_1 \pm \alpha_2 = 0$. Hence $\alpha_1 = \alpha_2 = 0$. For any other $x,y \in \mathscr{S}$, a similiar argument holds. 
\end{proof}

\begin{lemma}\label{lemma:typeIV-shortest-vectors-not-orthogonal}
There exist
    \[b_1, b_2 \in \mathscr{S}=\set{v_1, v_1', v_1'', v_1''',  \hat{w}_1, \hat{w}_2}\]
such that $\lambda_1(\Lambda_K)=\norm{b_1}$ and $\lambda_2(\Lambda_K)= \norm{b_2}$.
\end{lemma}

\begin{proof}
Let $x \in \Lambda_K$, $x \neq 0$. Let $u_1, u_2, u_3$ be as in~\eqref{eq:2}.  Then $\norm{x}^2 = \frac{1}{4}\sum_{i=1}^3 u_i^2 \norm{\hat{w}_i}^2$ since $\hat{w}_i \perp \hat{w}_j$ for all $i \neq j$. We wish to minimize $\norm{x}$. It follows from Lemma~\ref{lemma:typeIVC=H} and Lemma~\ref{lemma:typeIVC=S} that we wish to find the smallest integers $(\abs{u_1},\abs{u_2}, \abs{u_3})$ satisfying $u_1 \equiv u_2 \equiv u_3 \mod 2$. These are given by
\[(1,1,1),\ (2,0,0),\ (0,2,0),\ \text{and }(0,0,2).\]
In particular, $\left(\abs{u_1}, \abs{u_2}, \abs{u_3}\right) = (1,1,1)$ gives $\pm v_1$, $\pm v_1'$, $\pm v_1''$, $\pm v_1'''$. Similarly, $(2,0,0)$, $(0,2,0)$, $(0,0,2)$ give $\pm \hat{w}_1$, $\pm \hat{w}_2$, and $\pm \hat{w}_3$, respectively. Since $\hat{w}_i \perp \hat{w}_j$ for $i \neq j$, $\norm{v_1}^2 = \frac{1}{4}\left(\norm{\hat{w}_1}^2 +  \norm{\hat{w}_2}^2 + \norm{\hat{w}_3}^2\right)$.  The assumption $\norm{\hat{w}_1} < \norm{\hat{w}_2} < \norm{\hat{w}_3}$ implies that
\[\frac{1}{4}\left(\norm{\hat{w}_1}^2 + \norm{\hat{w}_2}^2 + \norm{\hat{w}_3}^2\right) < \frac{3}{4}\norm{\hat{w}_3}^2 < \norm{\hat{w}_3}^2.\]
Note that $\norm{v_1} = \norm{v_1'} = \norm{v_1''} = \norm{v_1'''}$. 
It follows that $\hat{w}_3$ is longer than every element of $\mathscr{S}$.
By Lemma~\ref{lemma:typeIV-pairwise-lin-ind} any pair of distinct vectors in $\mathscr{S}$ is linearly independent, whence the claim follows.
\end{proof}

\begin{theorem}\label{thm:typeIV-not-orthogonal}
Suppose $K$ is a real biquadratic field of type IV. Then $\Lambda_K$ is not orthogonal.
\end{theorem}

\begin{proof}
    Suppose otherwise that $\Lambda_K$ has orthogonal basis $B=\set{b_1,b_2,b_3}$ with $\norm{b_1} \leq \norm{b_2} \leq \norm{b_3}$. Then by Proposition~\ref{prop:orthog-lattice-lambda_i}, $\lambda_i(\Lambda_K)=\norm{\pm b_i}$, and by Lemma~\ref{lemma:typeIV-shortest-vectors-not-orthogonal},
    \[\pm b_1, \pm b_2 \in \mathscr{S} = \{v_1, v'_1, v''_1, v'''_1, \hat{w}_1, \hat{w}_2\}.\]
    Since $\norm{\hat{w}_1} < \norm{\hat{w}_2} < \norm{\hat{w}_3}$, an easy calculation shows that no pair of vectors from $\mathscr{S}$ is orthogonal, with the possible exception of
    \begin{align*}
      \hat{w}_1, \hat{w}_2; \\
      v_1, v_1'''; \textrm{ and} \\
      v_1', v_1''.
    \end{align*}
    If $b_1 = \pm \hat{w}_1, b_2 = \pm \hat{w}_2$, then by orthogonality of $B$, $b_3 = \pm \hat{w}_3$. But then $\Lambda_K \neq \Lambda(b_1,b_2,b_3)$.

    Suppose that $b_1,b_2 = \pm v_1, \pm v_1'''$; as $\norm{v_1} = \norm{v_1'''}$, the order does not matter. The orthogonality of $B$ implies that
    \[
      b_3 = a \hat{w}_1 - b\hat{w}_2
    \]
    where, if $c = \frac{a}{b}$, 
    \begin{align*}
      c &= \frac{\norm{\hat{w}_2}^2}{\norm{\hat{w}_1}^2} \\
      &= \frac{(\log |\epsilon_i|)^2}{(\log |\epsilon_j|)^2}
    \end{align*}
    for some choice of $i,j$ satisfying $1 \leq i,j \leq 3, i \neq j$. By Lemma~\ref{lemma:wi-orthogonal-lengths}, $c \notin \Q$. But then $b_3 \notin \Lambda_K$, from which we obtain a contradiction.

    The case of $v_1', v_1''$ is similar. The theorem follows.      
\end{proof}

\section{Cubic cyclic lattices}
\label{sec:cubic-cycl-latt}

Define a rank $2$ lattice to be \emph{equilateral triangular} if it has a basis $v_1, v_2$ for which $\norm{v_1} = \norm{v_2}$ and the angle between $v_1$ and $v_2$ is $\pi/3$. The prototypical equilateral triangular lattice is $\Lambda((1,0), (\frac{1}{2}, \frac{\sqrt{3}}{2})) \subset \R^2$. Note that all equilateral triangular lattices are similar to each other.
\begin{theorem}\label{thm:cubic-cyclic}
  If $K$ is a cubic Galois extension of $\Q$, then the log-unit lattice of $K$ is equilateral triangular.
\end{theorem}
A consequence of the theorem is that all such log-unit lattices are similar---that is, if $K, L$ are two cubic Galois extensions of $\Q$, then $\Lambda_K$ is isometric to $c\Lambda_L$ for some constant $c$.

The above result is reminiscent of results of~\cite{terr} (esp. Prop.~6.11), \cite{bhargava-harron2016}, and~\cite{1908.03969}, which instead describe the shape of the ring of integers embedded as a lattice under the Minkowski embedding.

\begin{proof}
  The Galois group of $K/\Q$ must be isomorphic to $\Z/3\Z$, and $K$ is a totally real field. By Dirichlet's Unit Theorem, $\Lambda_K$ is a rank $2$ lattice lying in the hyperplane $H \subset \R^3$ defined by $x_1 + x_2 + x_3 = 0$. Fix an embedding of $K$ in $\R$, and let $\sigma$ be a generator of the Galois group of $K/\Q$. Then $\Lambda_K$ is the image of
\begin{align*}
  \Log: \okx &\longrightarrow \R^3 \\
  \epsilon &\mapsto (\log |\epsilon|, \log |\sigma(\epsilon)|, \log |\sigma^2(\epsilon)|).
\end{align*}
The action of the Galois group on $\okx$ induces an action on $\Lambda_K$ via $\sigma(\Log(x)) = \Log(\sigma(x))$. The induced action of $\sigma$ is the same as that induced by rotating $H$ by $2\pi/3$ about the axis passing through the origin and $(1,1,1)$. In particular, $\Lambda_K$ is invariant under this rotation. As $\Lambda_K$ is a group, it is invariant under negation, and hence is invariant under rotation of $H$ by $\pi$. Therefore $\Lambda_K$ is invariant under rotation of $H$ by $\pi/3$.

One can now invoke Proposition 5.4.11, exercise 5.21, and exercise 5.22 in~\cite{artin}. These respectively describe the group of origin-fixing symmetries of a two-dimensional lattice, and show that in the two cases where the latter group has order divisible by $3$, the lattice is equilateral triangular. But we give a direct proof. Let $v_1$ be a shortest nonzero vector of $\Lambda_K$, and let $v_2$ be the image of $v_1$ upon rotating $H$ by $\pi/3$. Certainly $\norm{v_1} = \norm{v_2}$, so it suffices to show that $\Lambda_K = \Lambda(v_1,v_2)$. Suppose not. Let $x \in \Lambda_K$, $x \notin \Lambda(v_1, v_2)$. By translating by an appropriate element of $\Lambda(v_1, v_2)$, we may assume that
\[
  x = \alpha_1 v_1 + \alpha_2 v_2
\]
for some $0 \leq \alpha_1, \alpha_2 < 1$. Thus $x$ lies in the rhombus with vertices $0, v_1, v_2, v_1 + v_2$. Let $r = \norm{v_1} = \norm{v_2}$. Construct two circles of radius $r$ centered at $0$ and at $v_1 + v_2$. Since $x \neq v_1, v_2$, $x$ lies in the interior of one of these circles. If $x$ lies in the circle centered at $0$, then $\norm{x} < r$. If $x$ lies in the circle centered at $v_1 + v_2$, then $\norm{v_1+v_2 - x} < r$. In either case, we have found a nonzero element of $\Lambda_K$ which is shorter than $v_1$, contradicting our choice of $v_1$.
\end{proof}

\bibliographystyle{alphaurl}
\bibliography{refs}
\end{document}